\documentclass[12pt]{article}

\usepackage[english]{babel}
\usepackage[letterpaper,top=1.5cm,bottom=1.5cm,left=2.0cm,right=1.5cm,marginparwidth=1.75cm]{geometry}

\usepackage{amsmath,amssymb, bm, amsthm, bbm,dsfont}
\usepackage{graphicx}
\usepackage{newtxtext,newtxmath} 

\usepackage{natbib}
\usepackage[colorlinks=true,
linkcolor=blue,
citecolor=blue,
urlcolor=blue]{hyperref}
\usepackage{setspace}
\usepackage{float}
\usepackage{amsfonts}

\usepackage{xcolor}
\usepackage{multirow}

\newtheorem{theo}{Theorem}[section]
\newtheorem{defi}[theo]{Definition}
\newtheorem{remark}[theo]{Remark}
\newtheorem{proposition}[theo]{Proposition}

\newtheorem{theorem}{Theorem}

\usepackage{soul}

\usepackage{lineno}

\title{\bf Conditional copula representations and extremal bounds for multivariate statistical functionals}
\usepackage{authblk} 

 \author[1] { Roberto Vila \thanks{rovig161@unb.br }}
 \author[1]{Cira E G Otiniano \thanks{cira@unb.br }}
 \author[1]{Carolyne Brito  \thanks{carolynesb01@gmail.com}}
 \author[1]{Enzo Brasil  \thanks{enzoportobrasil@gmail.com}}

  \affil[1]{\normalfont 
 	Department of Statistics, University of
 	Bras\'ilia, Bras\'ilia, 70.910-900, Brazil}

\begin{document}
\maketitle
\begin{abstract}
In this paper, we derive a conditional copula representation for expectations of the form $\mathbb{E}[g(\boldsymbol{X})]$, where $\boldsymbol{X}$ is a random vector with arbitrary marginal distributions and $g$ is a measurable function satisfying suitable integrability conditions. The proposed representation explicitly separates the contributions of the marginal distributions and the dependence structure through conditional copula distributions, yielding a unified quantile--copula framework for a broad class of statistical functionals. This framework encompasses numerous quantities of practical interest, including moments, probabilities, dependence measures, inequality indices, entropy measures, and multivariate functionals. We further establish extremal bounds under fixed marginals by exploiting the concordance order on copulas and characterize the classes of functions for which these bounds apply through the notion of $\Delta$-antitonicity. Finally, several illustrative examples illustrate the versatility of the proposed framework through applications to risk measures, stochastic superiority probabilities, information measures, and option pricing under dependence uncertainty.
\end{abstract}



\section{Introduction}

Many statistical quantities can be expressed as expectations of measurable functions of a random vector. Examples include moments, probabilities, dependence measures, inequality indices, entropy measures, and numerous risk and reliability functionals. Although these quantities depend jointly on the marginal distributions and the dependence structure, their mathematical representations often fail to separate these two components explicitly.

Copula theory provides a natural framework for describing multivariate dependence independently of the marginal distributions through Sklar's theorem \cite{Sklar-1959}. This decomposition has become a fundamental tool in probability, statistics, finance, actuarial science, reliability, and quantitative risk management. Nevertheless, relatively few general representations are available for expectations of arbitrary measurable functions in terms of conditional copula distributions.

The first contribution of this paper is to derive a conditional copula representation for
$\mathbb E[g(\boldsymbol X)]$
that separates the marginal and dependence components through conditional copula distributions. The resulting representation is valid for arbitrary 
marginals and measurable functions satisfying mild integrability conditions. It provides a unified framework encompassing numerous statistical quantities as particular cases, including moments, probabilities, dependence measures, inequality indices, entropy measures, and several multivariate functionals.

The second contribution is to establish extremal bounds for these expectations under fixed marginal distributions. By combining the proposed representation with the concordance order on copulas and recent results on $\Delta$-antitonic functions due to \cite{Lux2017}, we derive explicit upper bounds for arbitrary dimensions and explicit lower bounds in the bivariate case. Moreover, we identify broad classes of statistical functionals for which these bounds apply directly.

The remainder of the paper is organized as follows. Section \ref{sect:2} develops the conditional copula representation and illustrates its applicability through several families of statistical functionals. Section~\ref{sect:3} establishes extremal dependence bounds under fixed marginals and classifies the corresponding functions according to their $\Delta$-antitonicity properties. Section \ref{sect:4} presents an illustrative example demonstrating the practical implementation of the proposed framework. Finally, Section \ref{sect:5} concludes the paper with a summary of the main findings and directions for future research.

\section{A conditional copula representation for $\mathbb E[g(\pmb{X})]$} \label{sect:2}

Let $\boldsymbol{X}=(X_1,\ldots,X_n)$ be a random vector with marginal distribution functions
$F_1,\ldots,F_n$, corresponding quantile functions
$F_1^{-1},\ldots,F_n^{-1}$, and survival functions
$\overline F_i=1-F_i$, $i=1,\ldots,n$.
Assume that the joint distribution of $\boldsymbol{X}$ is characterized by the copula $C$.

The associated survival copula is defined by
\[
\overline{C}(u_1,\ldots,u_n)
=
1
+
\sum_{k=1}^{n}
(-1)^k
\sum_{\substack{J\subseteq\{1,\ldots,n\}\\|J|=k}}
C(w_1^{(J)},\ldots,w_n^{(J)}),
\]
where 
$
w_i^{(J)}
=
1-u_i \mathds{1}_{J}(i),
$
with $\mathds{1}_{A}$ denoting the indicator function of a set 
$A$.

In particular,
$
\overline C(u,v)
=
u+v-1+C(1-u,1-v),
$
and
\begin{align*}
	\overline C(u,v,w)
	&=
	1-u-v-w
	\\[0,2cm]
	&+C(1-u,1-v,1)
	+C(1-u,1,1-w)
	+C(1,1-v,1-w)
	-C(1-u,1-v,1-w).
\end{align*}

Let
$
U_i=F_i(X_i),\ i=1,\ldots,n,
$
so that
$
(U_1,\ldots,U_n)\sim C.
$
Further, let
$
\boldsymbol U=(U_1,\ldots,U_{n-1})
$
and
$\boldsymbol u=(u_1,\ldots,u_{n-1})$. 
The conditional copula distribution of $U_n$ given $\boldsymbol U$ is defined by
\[
C_{n\mid1,\ldots,n-1}
(u_n\mid \boldsymbol u)
=
\frac{\partial^{\,n-1}}
{\partial u_1\cdots\partial u_{n-1}} \, 
C(u_1,\ldots,u_n),
\]
which coincides with
$
\mathbb P
\!\left(
U_n\leqslant u_n
\,\middle|\,
\boldsymbol U=\boldsymbol u
\right).
$

Throughout this paper, for any $d$-dimensional copula $D$, we denote by
$
\mathbb E_D[Y]
=
\int_{[0,1]^d}
Y(\boldsymbol t)\,{\rm d}D(\boldsymbol t)
$
the expectation of a measurable function
$Y:[0,1]^d\to\mathbb R$
with respect to the probability measure induced by $D$, provided that
$\mathbb E_D|Y|<\infty$, where
$\boldsymbol t=(t_1,\ldots,t_d)\in[0,1]^d$.
Likewise, if $D(\cdot\mid\boldsymbol t)$ denotes a conditional copula distribution, we write
$
\mathbb E_{D(\cdot\mid\boldsymbol t)}[Y]
=
\int_0^1
Y(v)\,{\rm d}D(v\mid\boldsymbol t),
$
for every measurable function
$Y:[0,1]\to\mathbb R$
such that
$\mathbb E_{D(\cdot\mid\boldsymbol t)}|Y|<\infty$.

Now let
$g:\mathbb R^n\to\mathbb R$
be measurable such that
$\mathbb E|g(\boldsymbol X)|<\infty$.
Applying the law of total expectation, we obtain
\begin{align*}
	\mathbb E[g(\boldsymbol X)]
	&=
	\mathbb E_C
	\!\left[
	g\!\left(
	F_1^{-1}(U_1),\ldots,F_n^{-1}(U_n)
	\right)
	\right]
	\\[0.2cm]
	&=
	\mathbb E_{C_{1,\ldots,n-1}}
	\!\left[
	\mathbb E_{C_{n\mid1,\ldots,n-1}(\cdot\mid\boldsymbol U)}
	\!\left[
	g\!\left(
	F_1^{-1}(U_1),\ldots,F_n^{-1}(U_n)
	\right)
	\right]
	\right]
=
	\int_{[0,1]^{n-1}}
	h_g(\boldsymbol u)\,
	{\rm d}C_{1,\ldots,n-1}(\boldsymbol u),
\end{align*}
where
$C_{1,\ldots,n-1}$ denotes the marginal copula of $\boldsymbol U$, and
\[
\begin{aligned}
	h_g(\boldsymbol u)
	&=
	\mathbb E_{C_{n\mid1,\ldots,n-1}(\cdot\mid\boldsymbol u)}
	\!\left[
	g\!\left(
	F_1^{-1}(u_1),\ldots,
	F_{n-1}^{-1}(u_{n-1}),
	F_n^{-1}(U_n)
	\right)
	\right]
	\\[0.2cm]
	&=
	\int_0^1
	g\!\left(
	F_1^{-1}(u_1),\ldots,
	F_{n-1}^{-1}(u_{n-1}),
	F_n^{-1}(u_n)
	\right)
	{\rm d}C_{n\mid1,\ldots,n-1}(u_n\mid\boldsymbol u).
\end{aligned}
\]

The above representation provides a unified quantile--copula framework for a broad class of statistical functionals. The function $h_g(\boldsymbol u)$ represents the conditional contribution of the quantile levels $(u_1,\ldots,u_{n-1})$ to the functional $\mathbb E[g(\boldsymbol X)]$, whereas the dependence structure is entirely encoded by the conditional copula distribution $C_{n\mid1,\ldots,n-1}$. Consequently, many commonly used statistical quantities admit representations of this form, as illustrated in Tables~\ref{Table-1}--\ref{Table-5}.
\begin{table}[H]
\centering
\caption{Moment-based statistics.}
\renewcommand{\arraystretch}{1.25}
\begin{tabular}{llll}
\hline
Quantity & $g(x,y)$ & $h_g(u)$\\
\hline

$\mathbb E[XY]$
&
$xy$
&
$\displaystyle
F_1^{-1}(u)
\int_0^1
F_2^{-1}(v)
\,{\rm d}C_{2|1}(v|u)
$
\\[3mm]

$\operatorname{Cov}(X,Y)$
&
$\displaystyle xy-\mu_X \mu_Y$
&
$\displaystyle
F_1^{-1}(u)
\int_0^1
F_2^{-1}(v)
\,{\rm d}C_{2|1}(v|u)
-
\int_0^1F_1^{-1}(v)\,{\rm d}v
\int_0^1F_2^{-1}(v)\,{\rm d}v
$
\\[3mm]

$\mathbb E\left[\dfrac{X}{Y}\right]$
&
$\displaystyle \frac{x}{y}$
&
$\displaystyle
F_1^{-1}(u)
\int_0^1
\frac{1}{F_2^{-1}(v)}
\,{\rm d}C_{2|1}(v|u)
$
\\[3mm]

$\mathbb E\!\left[\dfrac{X}{X+Y}\right]$
&
$\displaystyle \frac{x}{x+y}$
&
$\displaystyle
\int_0^1
\frac{F_1^{-1}(u)}
     {F_1^{-1}(u)+F_2^{-1}(v)}
\,{\rm d}C_{2|1}(v|u)
$
\\[3mm]

$\mathbb E\!\left[\dfrac{Y}{X+Y}\right]$
&
$\displaystyle \frac{y}{x+y}$
&
$\displaystyle
\int_0^1
\frac{F_2^{-1}(v)}
     {F_1^{-1}(u)+F_2^{-1}(v)}
\,{\rm d}C_{2|1}(v|u)
$
\\[3mm]

$\mathbb E|X-Y|$
&
$|x-y|$
&
$\displaystyle
\int_0^1
|F_1^{-1}(u)-F_2^{-1}(v)|
\,{\rm d}C_{2|1}(v|u)
$
\\[3mm]

$\mathbb E|X-Y|^r$
&
$|x-y|^r$
&
$\displaystyle
\int_0^1
|F_1^{-1}(u)-F_2^{-1}(v)|^r
\,{\rm d}C_{2|1}(v|u)
$
\\[3mm]

\hline
\end{tabular}
\label{Table-1}
\end{table}

\begin{table}[H]
\centering
\caption{Probability-based statistics.}
\renewcommand{\arraystretch}{1.25}
\begin{tabular}{llll}
\hline
Quantity & $g(x,y)$ & $h_g(u)$\\
\hline

$\mathbb P(X\leqslant Y)$
&
$\mathds 1_{\{x\leqslant y\}}$
&
$
1-
C_{2|1}
\!\left(
F_2(F_1^{-1}(u))
\,\middle|\,
u
\right)
$
\\

$\mathbb P(X+Y\leqslant t)$
&
$\mathds 1_{\{x+y\leqslant t\}}$
&
$
C_{2|1}
\!\left(
F_2(t-F_1^{-1}(u))
\,\middle|\,
u
\right)
$
\\

$\mathbb P(X-Y\leqslant t)$
&
$\mathds 1_{\{x-y\leqslant t\}}$
&
$\displaystyle
1-
C_{2|1}
\!\left(
F_2(F_1^{-1}(u)-t)
\,\middle|\,
u
\right)
$
\\

$\mathbb P(X\leqslant s,Y\leqslant t)$
&
$\mathds 1_{\{x\leqslant s,y\leqslant t\}}$
&
$\displaystyle
\mathds 1_{\{u\leqslant F_1(s)\}}
C_{2|1}(F_2(t)|u)
$
\\

$\mathbb P(XY\leqslant t)$
&
$\mathds 1_{\{xy\leqslant t\}}$
&
$\mathds 1_{ \{u\leqslant F_1(0)\} }
+
\bigl(1-2\mathds 1_{\{u\leqslant F_1(0)\}}\bigr)
\,C_{2|1}
\!\left(
F_2\!\left(
\frac{t}{F_1^{-1}(u)}
\right)
\,\middle|\,u
\right)
$
\\
$\mathbb P\left(\dfrac{X}{Y}\leqslant t\right)$
&
$\mathds 1_{\{{x\over y}\leqslant t\}}$
&
$\mathds 1_{\{u\leqslant F_1(0)\}}
+
\bigl(1-2\mathds 1_{\{u\leqslant F_1(0)\}}\bigr)
\,C_{2|1}
\!\left(
F_2\!\left(
\frac{F_1^{-1}(u)}{t}
\right)
\,\middle|\,u
\right)
$
\\
$\mathbb P\left(\dfrac{X}{X+Y}\leqslant t\right)$
&
$\mathds 1_{\{\frac{x}{x+y}\leqslant t\}}$
&
$            \overline C_{2|1} \left(\overline{F}_2 \left(
			\frac{1-t}{t}\,
			\overline{F}_1^{-1}(u)
			\right) \,\middle| 1-u \right). 
		$
\\[2mm]
\hline
\end{tabular}
\label{Table-2}
\end{table}

\begin{table}[H]
\centering
\caption{Dependence and inequality measures.}
\renewcommand{\arraystretch}{1.2}
\begin{tabular}{llll}
\hline
Measure & $g(x,y)$ & $h_g(u)$\\
\hline

Spearman's $\rho_s$
&
$12F_1(x)F_2(y)-3$
&
$\displaystyle
12u
\int_0^1
v\,{\rm d}C_{2|1}(v|u)
-3
$
\\[2mm]

Kendall's $\tau$
&
$4C(F_1(x),F_2(y))-1$
&
$\displaystyle
4
\int_0^1
C(u,v)\,
{\rm d}C_{2|1}(v|u)
-1
$
\\[2mm]

$\operatorname{Cov}(X,F_2(Y))$
&
$\displaystyle xF_2(y)-\frac{\mu_X}{2}$
&
$\displaystyle
F_1^{-1}(u)
\int_0^1
v\,{\rm d}C_{2|1}(v|u)
-
{1\over 2} \int_0^1F_1^{-1}(v)\,{\rm d}v
$
\\[4mm]

Gini correlation 
$\Gamma_{X,Y}$
&
$\dfrac{\displaystyle xF_2(y)-\frac{\mu_X}{2}}{\operatorname{Cov}(X,F_1(X))}$
&
$
\dfrac{\displaystyle  
F_1^{-1}(u)
\int_0^1
v\,{\rm d}C_{2|1}(v|u) -{1\over 2}\int_0^1F_1^{-1}(v)\,{\rm d}v }{\displaystyle
\int_0^1
uF_1^{-1}(u)\,{\rm d}u
-{1\over 2}\int_0^1F_1^{-1}(v)\,{\rm d}v}$
\\[9mm]

Bivariate Gini index 
$G(X,Y)$
&
$\displaystyle {|x-y|\over \mu_X+\mu_Y}$
&
$\displaystyle
\dfrac{\displaystyle
\int_0^1
|F_1^{-1}(u)-F_2^{-1}(v)|
\,{\rm d}C_{2|1}(v|u)
}{\displaystyle
\int_0^1F_1^{-1}(v)\,{\rm d}v
+
\int_0^1F_2^{-1}(v)\,{\rm d}v
}
$
\\[8mm]

\hline
\end{tabular}
\label{Table-3}
\end{table}

\begin{remark}
The tail dependence coefficients $\lambda_U$ and $\lambda_L$ do not admit a representation of the form
$\mathbb E[g(X,Y)]$.
Nevertheless, they can be expressed directly through the conditional copula
distribution $C_{2|1}$ as
\[
\lambda_U
=
\lim_{q\to1^-}
\frac{
1-2q+\int_0^q C_{2|1}(q|u)\,{\rm d}u
}{1-q}
\quad 
\text{and}
\quad 
\lambda_L
=
\lim_{q\to0^+}
\frac{
\int_0^q C_{2|1}(q|u)\,{\rm d}u
}{q}.
\]
\end{remark}

\begin{table}[H]
\centering
\caption{Entropy and information measures.}
\renewcommand{\arraystretch}{1.2}
\begin{tabular}{llll}
\hline
Measure & $g(x,y)$ & $h_g(u)$ \\
\hline

Shannon entropy 
$H(X)$
&
$-\log({f}_1(x))$ 
&
$\displaystyle
-\log\left({f}_1(F_1^{-1}(u))\right)
$
\\[2mm]

Joint Shannon entropy 
$H(X,Y)$
&
$-\log\left(f_{XY}(x,y)\right)$
&
$\displaystyle
-\int_0^1
\log\left(
f_{XY}
(F_1^{-1}(u),F_2^{-1}(v))\right)
\,{\rm d}C_{2|1}(v|u)
$
\\[2mm]

Mutual information 
$I(X,Y)$
&
$\displaystyle
\log\left(
\frac{
f_{XY}(x,y)
}{
f_1(x)f_2(y)
}
\right)
$
&
$\displaystyle
\int_0^1
\log
\left(
\frac{
f_{XY}(F_1^{-1}(u),F_2^{-1}(v))
}{
f_1(F_1^{-1}(u))
f_2(F_2^{-1}(v))
}
\right)
\,{\rm d}C_{2|1}(v|u)
$
\\[4mm]

\hline
\end{tabular}
\label{Table-4}
\end{table}


\begin{table}[H]
	\centering
	\caption{Examples of statistical quantities represented by
		$\mathbb E[g(X,Y,Z)]$.}
	\renewcommand{\arraystretch}{1.2}
	\begin{tabular}{p{2.5cm}p{4.0cm}p{9.0cm}}
		\hline
		Quantity & $g(x,y,z)$ & $h_g(u,v)$\\[1mm]
		\hline
		
		Expected maximum
		&
		$\max\{x,y,z\}$
		&
		$\displaystyle
		\mathbb E_{C_{3|12}}
		\!\left[
		\max\{F_1^{-1}(u),F_2^{-1}(v),Q(u,v)\}
		\right]
		$
		\\
		
		Expected minimum
		&
		$\min\{x,y,z\}$
		&
		$\displaystyle
		\mathbb E_{C_{3|12}}
		\!\left[
		\min\{F_1^{-1}(u),F_2^{-1}(v),Q(u,v)\}
		\right]
		$
		\\
		
		Expected range
		&
		$\max\{x,y,z\}-\min\{x,y,z\}$
		&
		$\displaystyle
		\mathbb E_{C_{3|12}}
		\!\left[
		\max\{\cdot\}-\min\{\cdot\}
		\right]
		$
		\\
		
		Pairwise absolute difference
		&
		$|x-y|+|x-z|+|y-z|$
		&
		$\displaystyle
		\mathbb E_{C_{3|12}}
		\!\left[
		|F_1^{-1}(u)-F_2^{-1}(v)|
		+
		|F_1^{-1}(u)-Q|
		+
		|F_2^{-1}(v)-Q|
		\right]
		$
		\\
		
		Ordering probability
		&
		$\mathds{1}_{\{x<y<z\}}$
		&
		$\displaystyle
		\mathds{1}_{\{F_1^{-1}(u)<F_2^{-1}(v)\}}
		\,\mathbb{P}_{C_{3|12}}
		\!\left(
		Q>F_2^{-1}(v)
		\right)
		$
		\\
		
		Expected product
		&
		$xyz$
		&
		$\displaystyle
		F_1^{-1}(u)F_2^{-1}(v)
		\,
		\mathbb E_{C_{3|12}}[Q]
		$
		\\
		
		Expected $L_p$ norm
		&
		$\left(x^p+y^p+z^p\right)^{1/p}$
		&
		$\displaystyle
		\mathbb E_{C_{3|12}}
		\!\left[
		\Bigl\{
		(F_1^{-1}(u))^p
		+(F_2^{-1}(v))^p
		+Q^p
		\Bigr\}^{1/p}
		\right]
		$
		\\
		
		\hline
	\end{tabular}
	
	\vspace{0.2cm}
	
	{\footnotesize
		Here
		$Q(u,v)=F_3^{-1}(U_3)$,
		where
		$U_3\sim C_{3|12}(\cdot\mid u,v)$.
	}
	\label{Table-5}
\end{table}




\section{Extremal bounds and quantile representations} \label{sect:3}

This section establishes extremal bounds and quantile representations for multivariate expectations through copula-based arguments.

Let
$
\boldsymbol X=(X_1,\ldots,X_n)
$
be an $\mathbb R^n$-valued random vector with marginal distribution functions
$F_1,\ldots,F_n$ and copula $C$. By Sklar's theorem \citep{Sklar-1959},
$
F_{\boldsymbol X}(x_1,\ldots,x_n)
=
C\!\left(
F_1(x_1),\ldots,F_n(x_n)
\right),
\
(x_1,\ldots,x_n)\in\mathbb R^n.
$

Throughout this section, the marginal distributions are assumed to be fixed. Hence, for every Borel measurable function
$g:\mathbb R^n\to\mathbb R$ satisfying
$\mathbb E|g(\boldsymbol X)|<\infty$,
the functional
\begin{align}
	\pi_g(C)
	\equiv 
	\mathbb E_C[g(\boldsymbol X)]
	&=
	\int_{\mathbb R^n}
	g(x_1,\ldots,x_n)\,
	{\rm d}
	C\!\left(
	F_1(x_1),\ldots,F_n(x_n)
	\right)
	\nonumber
	\\[0.2cm]
	&=
	\int_{[0,1]^n}
	g\!\left(
	F_1^{-1}(u_1),\ldots,F_n^{-1}(u_n)
	\right)
	{\rm d}C(u_1,\ldots,u_n),
	\label{id-exp-cop}
\end{align}
depends only on the copula $C$.
\begin{defi}[\cite{Lux2017}]
	\label{antitonic-def}
	For $i=1,\ldots,n$ and $a_i<b_i$, define the finite difference operator
	\[
	\Delta^{\,i}_{a_i,b_i}
	g(x_1,\ldots,x_n)
	=
	g(x_1,\ldots,x_{i-1},b_i,x_{i+1},\ldots,x_n)
	-
	g(x_1,\ldots,x_{i-1},a_i,x_{i+1},\ldots,x_n).
	\]
	A function
	$g:\mathbb R^n\to\mathbb R$
	is called
	\emph{$\Delta$-antitonic}
	if, for every hyperrectangle
	$
	\prod_{i=1}^{n}[a_i,b_i]
	\subset
	\mathbb R^n,
	\
	a_i<b_i,
	$
	it holds that
	\[
	\left(
	\Delta^{\,1}_{a_1,b_1}
	\circ
	\cdots
	\circ
	\Delta^{\,n}_{a_n,b_n}
	\right)
	g(x_1,\ldots,x_n)
	\geqslant0,
	\quad
	\forall\,
	(x_1,\ldots,x_n)\in\mathbb R^n.
	\]
	The operator
	$\Delta$
	is referred to as the
	finite difference operator.
\end{defi}

\begin{remark}\label{rem-antitonic}
	Suppose that
	$g:\mathbb R^n\to\mathbb R$
	is of class
	$C^n$.
	Repeated application of the Fundamental Theorem of Calculus yields
	\[
	\left(
	\Delta^{\,1}_{a_1,b_1}
	\circ
	\cdots
	\circ
	\Delta^{\,n}_{a_n,b_n}
	\right)
	g(x_1,\ldots,x_n)
	=
	\int_{a_n}^{b_n}
	\cdots
	\int_{a_1}^{b_1}
	\frac{\partial^n g(x_1,\ldots,x_n)}
	{\partial x_1\cdots\partial x_n}
	\,{\rm d}x_1
	\cdots
	{\rm d}x_n.
	\]
	Consequently,
	$g$
	is
	$\Delta$-antitonic
	if and only if
	$
	{\partial^n g(x_1,\ldots,x_n)}/
	{\partial x_1\cdots\partial x_n}
	\geqslant0,
	\
	\forall\,
	(x_1,\ldots,x_n)\in\mathbb R^n.
	$
\end{remark}

To state the following results, let
\[
W_n(u_1,\ldots,u_n)
=
\max\!\left\{
\sum_{i=1}^{n}u_i-n+1,\,
0
\right\}
\quad
\text{and}
\quad
M_n(u_1,\ldots,u_n)
=
\min\{u_1,\ldots,u_n\},
\]
denote the lower and upper Fréchet--Hoeffding bounds, respectively.
It is well known that
$M_n$
is an $n$-copula for every
$n\geqslant2$,
whereas
$W_n$
is an $n$-copula only when
$n=2$
and a proper quasi-copula for
$n\geqslant3$
\citep{Genest1999}.

\begin{proposition}\label{antitonic}
	Let
	$g:\mathbb R^n\to\mathbb R$
	be
	$\Delta$-antitonic.
	Then the functional
	$\pi_g$
	is increasing with respect to the concordance order. Hence, for every
	$C\in\mathcal C_n$,
	\[
	\inf_{D\in\mathcal C_n}\pi_g(D)
	\leqslant
	\pi_g(C)
	\leqslant
	\pi_g(M_n),
	\]
	where $\mathcal C_n$ denotes the class of all $n$-copulas 
	
	Moreover, if
	$-g$
	is
	$\Delta$-antitonic,
	then
	\[
	\pi_g(M_n)
	\leqslant
	\pi_g(C)
	\leqslant
	\sup_{D\in\mathcal C_n}\pi_g(D).
	\]
	When $n=2$, the bounds are attained and
	$
	\pi_g(W_2)
	=
	\inf_{D\in\mathcal C_2}\pi_g(D)
	$
	in the first case, whereas
	$
	\pi_g(W_2)
	=
	\sup_{D\in\mathcal C_2}\pi_g(D)
	$
	in the second case.
\end{proposition}
\begin{proof}
	The result follows immediately from
	Theorem~5.5 and Proposition~6.1 of
	\cite{Lux2017}.
\end{proof}

\begin{theorem}\label{main-result-multivariate}
	Let
	$g:\mathbb R^n\to\mathbb R$
	be
	$\Delta$-antitonic.
	Then
	\[
	\inf_{D\in\mathcal C_n}\pi_g(D)
	\leqslant
	\mathbb E[g(\boldsymbol X)]
	\leqslant
	\int_0^1
	g\!\left(
	F_1^{-1}(u),
	\ldots,
	F_n^{-1}(u)
	\right)
	{\rm d}u.
	\]
	
	Moreover, if
	$-g$
	is
	$\Delta$-antitonic,
	then
	\[
	\int_0^1
	g\!\left(
	F_1^{-1}(u),
	\ldots,
	F_n^{-1}(u)
	\right)
	{\rm d}u
	\leqslant
	\mathbb E[g(\boldsymbol X)]
	\leqslant
	\sup_{D\in\mathcal C_n}\pi_g(D).
	\]
	
	In particular, when $n=2$, the lower (respectively, upper) bound is attained at the Fréchet--Hoeffding lower copula $W_2$.
\end{theorem}

\begin{proof}
	Since
	$
	\mathbb E[g(\boldsymbol X)]
	=
	\pi_g(C),
	$
	Proposition~\ref{antitonic} yields
	$$
	\mathbb E[g(\boldsymbol X)]
	\leqslant
	\pi_g(M_n).
	$$
	By \eqref{id-exp-cop},
	\[
	\pi_g(M_n)
	=
	\int_{[0,1]^n}
	g\!\left(
	F_1^{-1}(u_1),
	\ldots,
	F_n^{-1}(u_n)
	\right)
	{\rm d}M_n(u_1,\ldots,u_n).
	\]
	Since
	$
	M_n(u_1,\ldots,u_n)
	=
	\min\{u_1,\ldots,u_n\},
	$
	the probability measure induced by
	$M_n$
	is concentrated on the diagonal
	$
	\{(u,\ldots,u):0<u<1\}.
	$
Equivalently, \[ \frac{\partial^n M_n(u_1,\ldots,u_n)} {\partial u_1\cdots\partial u_n} = \delta_u(u_2)\cdots\delta_u(u_n) \] in the sense of distributions. Therefore, 
\[ 
\pi_g(M_n) = \int_0^1 g\!\left( F_1^{-1}(u), \ldots, F_n^{-1}(u) \right) {\rm d}u, 
\] 
which proves the first inequality.
	
	The second inequality follows by applying the first part to the function $-g$.
\end{proof}

\begin{remark}\label{rem-W2}
	For $n=2$, the lower Fréchet--Hoeffding bound
	$
	W_2(u,v)=\max\{u+v-1,0\}
	$
	is the countermonotonic copula. Hence,
	\[
	\pi_g(W_2)
	=
	\int_0^1
	g\!\left(
	F_1^{-1}(u),
	F_2^{-1}(1-u)
	\right)
	{\rm d}u,
	\]
	and Theorem~\ref{main-result-multivariate} becomes
	\[
	\int_0^1
	g\!\left(
	F_1^{-1}(u),
	F_2^{-1}(1-u)
	\right)
	{\rm d}u
	\leqslant
	\mathbb E[g(X,Y)]
	\leqslant
	\int_0^1
	g\!\left(
	F_1^{-1}(u),
	F_2^{-1}(u)
	\right)
	{\rm d}u.
	\]
\end{remark}

Theorem~\ref{main-result-multivariate} provides extremal bounds for
$\mathbb E[g(\boldsymbol X)]$
over the class of admissible dependence structures with fixed marginals.
The upper bound depends only on the marginal quantile functions and is attained under comonotonicity.
For $n=2$, the lower bound is also available in closed form and is attained under countermonotonicity. Together with the quantile--copula representation developed in the previous section, these results establish a unified framework for characterizing the set of attainable values of a broad class of multivariate statistical functionals.


Tables~\ref{tab:Delta-antitonic}--\ref{tab:Delta-antitonic-Table5}
classify the functions appearing in Tables~\ref{Table-1}--\ref{Table-5}
according to whether $g$ or $-g$ is $\Delta$-antitonic in the sense of
Definition~\ref{antitonic-def}. Consequently, they identify the statistical
quantities for which the upper or lower extremal bounds of
Theorem~\ref{main-result-multivariate} apply directly, thereby providing a
practical guide to the use of the proposed framework.
\begin{table}[H]
	\centering
	\caption{Classification of the functions in Table~\ref{Table-1} according to $\Delta$-antitonicity.}
	\renewcommand{\arraystretch}{1.2}
	\begin{tabular}{lcc}
		\hline
		Quantity & $g$ $\Delta$-antitonic & $-g$ $\Delta$-antitonic\\
		\hline
		$\mathbb E[XY]$ & Yes & No\\
		$\operatorname{Cov}(X,Y)$ & Yes & No\\
		$\mathbb E\left[\dfrac{X}{Y}\right]$ & No & Yes\\
		$\mathbb E\!\left[\dfrac{X}{X+Y}\right]$ & No & No\\
		$\mathbb E\!\left[\dfrac{Y}{X+Y}\right]$ & No & No\\
		$\mathbb E|X-Y|$ & No & Yes\\
		$\mathbb E|X-Y|^r,\;0<r<1$ & Yes & No\\
		$\mathbb E|X-Y|^r,\;r\geqslant1$ & No & Yes\\
		\hline
	\end{tabular}
	\label{tab:Delta-antitonic}
\end{table}

\begin{table}[H]
	\centering
	\caption{Classification of the functions in Table~\ref{Table-2} according to $\Delta$-antitonicity.}
	\renewcommand{\arraystretch}{1.2}
	\begin{tabular}{lcc}
		\hline
		Quantity &
		$g$ $\Delta$-antitonic &
		$-g$ $\Delta$-antitonic
		\\
		\hline
		
		$\mathbb P(X\leqslant Y)$
		&
		No
		&
		Yes
		\\
		
		$\mathbb P(X+Y\leqslant t)$
		&
		No
		&
		Yes
		\\
		
		$\mathbb P(X-Y\leqslant t)$
		&
		Yes
		&
		No
		\\
		
		$\mathbb P(X\leqslant s,Y\leqslant t)$
		&
		Yes
		&
		No
		\\
		
		$\mathbb P(XY\leqslant t)$
		&
		No
		&
		Yes
		\\
		
		$\mathbb P\!\left(\dfrac{X}{Y}\leqslant t\right)$
		&
		No
		&
		Yes
		\\
		
		$\mathbb P\!\left(\dfrac{X}{X+Y}\leqslant t\right)$
		&
		No
		&
		Yes
		\\
		
		\hline
	\end{tabular}
	\label{tab:Delta-antitonic-Table2}
\end{table}

\begin{table}[H]
	\centering
	\caption{Classification of the functions in Table~\ref{Table-3} according to $\Delta$-antitonicity.}
	\renewcommand{\arraystretch}{1.2}
	\begin{tabular}{lcc}
		\hline
		Measure &
		$g$ $\Delta$-antitonic &
		$-g$ $\Delta$-antitonic
		\\
		\hline
		
		Spearman's $\rho_s$
		&
		Yes
		&
		No
		\\
		
		Kendall's $\tau$
		&
		Yes
		&
		No
		\\
		
		$\operatorname{Cov}(X,F_2(Y))$
		&
		Yes
		&
		No
		\\
		
		Gini correlation $\Gamma_{X,Y}$
		&
		Yes
		&
		No
		\\
		
		Bivariate Gini index $G(X,Y)$
		&
		No
		&
		Yes
		\\
		
		\hline
	\end{tabular}
	\label{tab:Delta-antitonic-Table3}
\end{table}

\begin{table}[H]
	\centering
	\caption{Applicability of the $\Delta$-antitonicity criterion to the functions in Table~\ref{Table-4}.}
	\renewcommand{\arraystretch}{1.2}
	\begin{tabular}{lcc}
		\hline
		Measure &
		$g$ $\Delta$-antitonic &
		Remarks
		\\
		\hline
		Shannon entropy $H(X)$
		&
		Yes
		&
		Also $-g$ is $\Delta$-antitonic.
		\\
		Joint Shannon entropy $H(X,Y)$
		&
		Not in general
		&
		$g$ depends on $f_{XY}$.
		\\
		Mutual information $I(X,Y)$
		&
		Not in general
		&
		$g=\log(c)$ depends on the copula density.
		\\
		\hline
	\end{tabular}
\end{table}

\begin{table}[H]
	\centering
	\caption{Classification of the functions in Table~\ref{Table-5} according to $\Delta$-antitonicity.}
	\renewcommand{\arraystretch}{1.2}
	\begin{tabular}{lcc}
		\hline
		Quantity &
		$g$ $\Delta$-antitonic &
		$-g$ $\Delta$-antitonic
		\\
		\hline
		
		Expected maximum
		&
		Yes
		&
		No
		\\
		
		Expected minimum
		&
		No
		&
		Yes
		\\
		
		Expected range
		&
		No
		&
		No
		\\
		
		Pairwise absolute difference
		&
		Yes
		&
		Yes
		\\
		
		Ordering probability
		&
		Yes
		&
		No
		\\
		
		Expected product
		&
		Yes
		&
		No
		\\
		
		Expected $L_1$ norm
		&
		Yes
		&
		Yes
		\\
		
		Expected $L_p$ norm ($p>1$)
		&
		No general classification
		&
		No general classification
		\\
		
		\hline
	\end{tabular}
	\label{tab:Delta-antitonic-Table5}
\end{table}

\section{Illustrative examples} \label{sect:4}
Many fundamental quantities in probability and statistics can be expressed as expectations of measurable functions of a random vector, namely $\mathbb{E}[g(\boldsymbol X)]$, $\boldsymbol X=(X_1, \dots X_n)$. This broad class includes mixed and higher-order moments, probabilities of multivariate events, distributions of transformations of random vectors, conditional expectations, and information measures. 
For example, the quantile associated with the distribution of the aggregate variable
$
L=\sum_{i=1}^{n}X_i
$
 defines the Value-at-Risk (VaR), whereas the Expected Shortfall (ES) measures the expected magnitude of L beyond this quantile. 
 Assuming that \(L\) has a continuous distribution and 
$
g(x_1,\ldots,x_n)=
\left(\sum_{i=1}^{n}x_i\right)
\mathds{1}_{\left\{\sum_{i=1}^{n}x_i>
\operatorname{VaR}_{\alpha}(L)\right\}},
$
 the Expected Shortfall at confidence level \(\alpha\)  defined by
\[
ES_{\alpha}(L)=
\frac{1}{1-\alpha} \,
\mathbb{E}\!\left[
L \mathds 1_{\left\{L>\mathrm{VaR}_{\alpha}(L)\right\}}
\right]
\]
corresponds to the calculation of
$
ES_{\alpha}(L)
=
\mathbb{E}[g(\mathbf{X})]/(1-\alpha).
$
 
 These risk measures provide a natural framework for studying aggregate effects arising from multiple sources of uncertainty. In hydrology, analogous concepts are used to define multivariate return periods, which quantify the likelihood of occurrence of compound extreme events involving several dependent variables. Therefore, quantities such as VaR, ES, and multivariate return periods are fundamental tools for characterizing aggregate risks and extreme events in actuarial science, finance, and environmental applications \citep{McNeil2015,Salvadori2016}.
 
 More generally, expectations of the form 
$\mathbb{E}[g(\boldsymbol X)]$  encompass a wide variety of statistical functionals arising in probability, statistical inference, finance, actuarial science, reliability engineering, environmental sciences, and information theory.

The versatility of the proposed framework is illustrated even in the bivariate case. Indeed, suitable choices of the function $g$ recover several classical statistical functionals. For instance, $g(x,y)=\mathds 1_{\{x<y\}}$,
$
\mathbb{P}(X<Y)=\mathbb{E}\left[\mathds{1}_{\{X<Y\}}\right]
$
represents the probability of stochastic superiority and corresponds to the population quantity underlying the Mann–Whitney statistic  \citep{MannWhitney1947,Hollander2014}. This probability has also found practical applications; for example, \cite{OtinianoMaluf2019} employed it to compare river streamflows between the dry and rainy seasons. Likewise, the mixed moment $\mathbb{E}[XY]$ determines the covariance structure \citep{CasellaBerger2002}, while Shannon entropy and mutual information are themselves expectations of suitable functions \citep{CoverThomas2006}. These examples illustrate how a single representation encompasses statistical functionals that have traditionally been studied in different contexts, providing a unified analytical framework for theoretical developments and practical applications.


Theorem~\ref{main-result-multivariate} provides lower and upper bounds for risk measures such as 
$ES_{\alpha}(L)$ within a class of dependence structures consistent with fixed 
marginal distributions. Consequently, when the marginal distributions are known 
but the dependence structure among the components of the portfolio is unspecified, 
the proposed bounds represent best- and worst-case scenarios for the 
Expected Shortfall. Since $ES_{\alpha}(L)$ corresponds to the conditional mean 
of the aggregate variable $L$ given that it exceeds the quantile 
$\operatorname{VaR}_{\alpha}(L)$, these bounds allow us to quantify the impact of 
dependence uncertainty on the average severity of extreme portfolio losses 
\citep{McNeil2015, Embrechts2013}.

Another important application of Theorem~\ref{main-result-multivariate} is the classical problem of option pricing under dependence uncertainty. In this framework, the marginal distributions of the underlying assets are assumed to be known, whereas the copula describing their joint dependence is only partially identified, \cite{Lux2017}. Since option prices are expectations of payoff functions of the underlying assets, the proposed bounds yield best- and worst-case option prices over the class of admissible dependence structures with fixed marginals.

\section{Concluding remarks} \label{sect:5}

This paper introduced a conditional copula representation for expectations of the form
$
E[g(\boldsymbol{X})],
$
where $\boldsymbol{X}$ is a random vector with arbitrary marginal distributions and $g$ is a measurable function satisfying suitable integrability conditions. The proposed representation separates the contributions of the marginal distributions and the dependence structure through conditional copula distributions, providing a unified quantile--copula framework for expressing a broad class of statistical functionals.
We also established extremal bounds for these expectations under fixed marginal distributions by exploiting the concordance order on copulas. Furthermore, the characterization of statistical functionals through the notion of $\Delta$-antitonicity provides practical conditions under which the proposed bounds are applicable.
The illustrative examples demonstrate the applicability of the proposed framework to a variety of statistical functionals, including risk measures, stochastic superiority probabilities, information measures, and option pricing under dependence uncertainty.
Future research may consider extensions of the proposed framework to broader classes of dependence structures, including quasi-copulas and partially specified dependence models. Other directions include the derivation of extremal bounds for additional classes of statistical functionals and the development of computational methods for implementing the proposed representations in high-dimensional settings.


\section*{Declarations}

\paragraph*{Ethics Approval}
Not applicable.

\paragraph*{Funding Declaration}
This study was financed in part by CAPES (Finance Code 001). 

\paragraph*{Disclosure statement}
There are no conflicts of interest to disclose.

\paragraph{Data availability}
No data was used for the research described in the article.

\paragraph*{Author Contributions Statement}
All authors contributed equally to the conception and design of the study, data analysis, interpretation of the results, manuscript preparation, and revision. All authors read and approved the final version of the manuscript.


\end{document}